\documentclass[12pt, reqno]{amsart}
\usepackage{amsmath, amsthm, amscd, amsfonts, amssymb, graphicx, color}
\usepackage[bookmarksnumbered, colorlinks, plainpages]{hyperref}
\hypersetup{colorlinks=true,linkcolor=red, anchorcolor=green, citecolor=cyan, urlcolor=red, filecolor=magenta, pdftoolbar=true}

\newtheorem{theorem}{Theorem}[section]
\newtheorem{lemma}[theorem]{Lemma}

\theoremstyle{definition}

\theoremstyle{remark}

\numberwithin{equation}{section}

\begin{document}
\setcounter{page}{1}

\begin{center}
{\Large \textbf{Some approximation results on
Bernstein-Schurer operators defined by $(p,q)$-integers~~(Revised)
}}

\bigskip

\textbf{M. Mursaleen}, \textbf{Md. Nasiruzzaman} and \textbf{Ashirbayev Nurgali}

Department of\ Mathematics, Aligarh Muslim University, Aligarh--202002, India%
\\[0pt]

mursaleenm@gmail.com; nasir3489@gmail.com \\[0pt%
]

\bigskip

\bigskip

\textbf{Abstract}
\end{center}

\parindent=8mm {\footnotesize {In the present article, we have given a corrigendum to our paper �Some approximation results on
Bernstein-Schurer operators defined by $(p,q)$-integers� published in Journal of Inequalities and Applications (2015) 2015:249.}}

\bigskip


\parindent=0mm\emph{AMS Subject Classification (2010):} 41A10, 41A25, 41A36, 40A30.


\section{ \textbf{{Introduction and Preliminaries } }}

In 1912, S.N Bernstein \cite{berns} introduced the following sequence of
operators $B_n:C[0,1] \to C[0,1]$ defined for any $n \in \mathbb{N}
$ and for any $f \in C[0,1]$ such as
\begin{equation*}
B_n(f;x)=\sum_{k=0}^n \binom{n}{k}x^k(1-x)^{n-%
k} f\left(\frac{k}{n}\right),~~~~~~~x \in [0,1].~~~~~~~~~~~~~~~~~\eqno(1.1)
\end{equation*}

In 1987, Lupa� \cite{lupas2} introduced the $q$-Bernstein operators by
applying the idea of $q$-integers, and in 1997 another generalization of
these operators introduced by Philip \cite{philip}. Later on, many authors
introduced $q$-generalization of various operators and investigated several
approximation properties. For instance, $q$-analogue of Stancu-Beta
operators in \cite{aral2} and \cite{mur1}; $q$-analogue of
Bernstein-Kantorovich operators in \cite{radu}; $q$- Baskakov-Kantorovich
operators in \cite{gupta}; $q$-Sz$\acute{a}$sz-Mirakjan operators in \cite%
{kant}; $q$-Bleimann, Butzer and Hahn operators in \cite{aral1} and in \cite%
{ersan}; $q$-analogue of Baskakov and Baskakov-Kantorovich operators in \cite%
{mah1}; $q$-analogue of Sz$\acute{a}$sz-Kantorovich operators in \cite{mah2}%
; and $q$-analogue of generalized Bernstein-Shurer operators in \cite{mur3}.

We recall certain notations on $(p,q)$-calculus.

The $(p,q)$-integer was introduced in order to generalize or unify several
forms of $q$-oscillator algebras well known in the earlier physics
literature related to the representation theory of single parameter quantum
algebras \cite{chak}. The $(p,q)$-integer $[n]_{p,q}$ is defined by
\begin{equation*}
\lbrack n]_{p,q}=\frac{p^{n}-q^{n}}{p-q},~~~~~~~n=0,1,2,\cdots ,~~0<q<p\leq
1.
\end{equation*}%
The $(p,q)$-Binomial expansion is
\begin{equation*}
(ax+by)_{p,q}^{n}:=\sum\limits_{k=0}^{n}p^{\frac{(n-k)(n-k-1)}{2}}q^{\frac{k(k-1)}{2}}
\left[
\begin{array}{c}
n \\
k%
\end{array}%
\right] _{p,q}a^{n-k}b^{k}x^{n-k}y^{k}
\end{equation*}

\begin{equation*}
(x+y)_{p,q}^{n}:=(x+y)(px+qy)(p^{2}x+q^{2}y)\cdots (p^{n-1}x+q^{n-1}y).
\end{equation*}%
Also, the $(p,q)$-binomial coefficients are defined by

\begin{equation*}
\left[
\begin{array}{c}
n \\
k%
\end{array}%
\right] _{p,q}:=\frac{[n]_{p,q}!}{[k]_{p,q}![n-k]_{p,q}!}.
\end{equation*}%
Details on $(p,q)$-calculus can be found in \cite{vivek}. For $p=1$, all the
notions of $(p,q)$-calculus are reduced to $q$-calculus \cite{and}.

In 1962, Schurer \cite{shr} introduced and studied the operators $S_{m,\ell }:C[0,\ell+1]\rightarrow C[0,1]$ defined for any $m\in
\mathbb{N}$ and $\ell $ be fixed in $\mathbb{N}$ and any function $f\in
C[0,\ell +1]$ as follows
\begin{equation}\label{12}
S_{m,\ell }(f;x)=\sum_{k=0}^{m+\ell }\left[
\begin{array}{c}
m+\ell \\
k%
\end{array}%
\right]%
x^{k}(1-x)^{m+\ell -k}f\left( \frac{k}{m}\right) ,~~~~~~~x\in [0,1].
\end{equation}

For any $m \in$ and $f \in C[0,\ell+1]$, $~\ell$ is fixed, then $q-$ analogue of Bernstein-Schurer operators in \cite{mura} defined as follows

\begin{equation}\label{13}
\tilde{B}_{m,\ell}(f;q;x)=\sum\limits_{k=0}^{m+\ell}\left[
\begin{array}{c}
m+\ell \\
k%
\end{array}%
\right] _{q}x^{k}\prod\limits_{s=0}^{m+\ell-k-1}(1-q^{s}x)~~f\left( \frac{%
[k]_{q} }{[m]_{q}}\right) ,~~x\in [0,1].
\end{equation}

Our aim is to introduce a $(p,q)$-analogue of these operators. We
investigate the approximation properties of this class and we estimate the
rate of convergence and some theorem by using the modulus of continuity. We
study the approximation properties based on Korovkin's type approximation
theorem and also establish the some direct theorem.

\section{ \textbf{{Construction of $(p,q)$-Bernstein-Schurer operators (Revised)} }}
Mursaleen et. al \cite{mur4} has defined $(p,q)$-analogue of Bernstein operators as:
\begin{equation}\label{bn}
{B}_{n}^{p,q}(f;x)=\sum_{k=0}^{n} \left[
\begin{array}{c}
n \\
k%
\end{array}%
\right] _{p,q}x^k \prod_{s=0}^{n-k-1}(p^s-q^s x)
f\left(\frac{[k]_{p,q}}{[n]_{p,q}}\right),~~~~~~~x \in [0,1].
\end{equation}
But ${B}_{m,\ell}^{p,q}(f;x) \neq 1$, for all $x \in [0,1]$. Hence, They re-introduced the $(p,q)$ Bernstein operators \cite{arxive} as follows:\newline

\begin{equation}\label{noj}
{B}_{m,\ell}^{p,q}(f;x)=\frac{1}{p^{\frac{n(n-1)}{2}}}\sum_{k=0}^{n} \left[
\begin{array}{c}
n \\
k%
\end{array}%
\right] _{p,q}p^{\frac{k(k-1)}{2}}x^k \prod_{s=0}^{n-k-1}(p^s-q^s x)
f\left(\frac{[k]_{p,q}}{p^{k-n}[n]_{p,q}}\right),~~~~~~~x \in [0,1].
\end{equation}

Mursaleen et. al \cite{nt1} introduced the $(p,q)$-analogue of Bernstein Schurer operators as:
\begin{equation}\label{mn1}
{B}_{m,\ell}^{p,q}(f;x)=\sum_{k=0}^{m+\ell} \left[
\begin{array}{c}
m+\ell \\
k%
\end{array}%
\right] _{p,q}x^k \prod_{s=0}^{m+\ell-k-1}(p^s-q^s x)
f\left(\frac{[k]_{p,q}}{[m]_{p,q}}\right),~~~~~~~x \in [0,1].
\end{equation}
But ${B}_{m,\ell}^{p,q}(f;x) \neq 1$, for all $x \in [0,1]$. Hence, we re-define our operators as follows:\newline

 We consider  $0<q<p\leq 1$ and for any $m \in \mathbb{N},~~f \in C[0,\ell+1]$, $\ell$ is fixed, we construct a revised
generalized $(p,q)$-Bernstein Schurer operators:
\begin{equation}\label{nj}
{B}_{m,\ell}^{p,q}(f;x)=\frac{1}{p^{\frac{(m+\ell)(m+\ell-1)}{2}}}\sum_{k=0}^{m+\ell} \left[
\begin{array}{c}
m+\ell \\
k%
\end{array}%
\right] _{p,q}p^{\frac{k(k-1)}{2}}x^k \prod_{s=0}^{m+\ell-k-1}(p^s-q^s x)
f\left(\frac{[k]_{p,q}}{p^{k-m-\ell}[m]_{p,q}}\right),~~~~~~~x \in [0,1].
\end{equation}
Clearly, the operator defined by \eqref{nj} is linear and positive. And if we put
$p=1$ in \eqref{nj}, then $(p,q)$ Shurer operator given by \eqref{nj} turn out the $q$%
- Bernstein Shurer operators \cite{mura}. 

\begin{lemma}\label{bb}
Let ${B}_{m,\ell}^{p,q}(.;.)$ be given by \eqref{nj}, then for any $%
x \in [0,1]$ and $0<q<p\leq 1$ we have the following identities

\begin{enumerate}
\item[$($i$)$] ${B}_{m,\ell}^{p,q}(e_0;x)=1$

\item[$($ii$)$] ${B}_{m,\ell}^{p,q}(e_1;x)=\frac{[m+\ell]_{p,q}%
}{[m]_{p,q}}x$

\item[$($iii$)$] ${B}_{m,\ell}^{p,q}(e_2;x)=\frac{p^{m+\ell-1}[m+\ell]_{p,q}}{[m]_{p,q}^2}x+\frac{%
q[m+\ell]_{p,q}[m+\ell-1]_{p,q}}{[m]_{p,q}^2}x^{2}$

\end{enumerate}

where $e_j(t)=t^j,~~j=0,1,2, \cdots.$
\end{lemma}

\begin{proof}
\begin{enumerate}

\item[$($i$)$] For $0<q<p\leq 1$ we use the known identity from \cite{mur4}
\begin{equation*}
\sum_{k=0}^{n} \left[
\begin{array}{c}
n \\
k%
\end{array}%
\right] _{p,q}p^{\frac{k(k-1)}{2}}x^k
\prod_{s=0}^{n-k-1}(p^s-q^s x)=p^{\frac{(n)(n-1)}{2}},~~~x \in [0,1]
\end{equation*}
Suppose we choose $n=m+\ell.$\newline
Since
\begin{equation*}
(1-x)_{p,q}^{m+\ell-k}=\prod_{s=0}^{m+\ell-k-1}(p^s-q^s x),
\end{equation*}
we get
\begin{equation*}
\sum_{k=0}^{m+\ell} \left[
\begin{array}{c}
m+\ell \\
k%
\end{array}%
\right] _{p,q}p^{\frac{k(k-1)}{2}}x^k
\prod_{s=0}^{m+\ell-k-1}(p^s-q^s x)=p^{\frac{(m+\ell)(m+\ell-1)}{2}}
\end{equation*}
Consequently, which implies ${B}_{m,\ell}^{p,q}(e_0;x)=1$.

\item[$($ii$)$] Clearly we have
 \begin{eqnarray*}
{B}_{m,\ell}^{p,q}(e_1;x) &=&\frac1{p^{\frac{(m+\ell)(m+\ell-1)}2}}\sum\limits_{k=0}^{m+\ell}\left[
\begin{array}{c}
m+\ell\\
k%
\end{array}%
\right] _{p,q}p^{\frac{k(k-1)}2}x^{k}\prod\limits_{s=0}^{m+\ell-k-1}(p^{s}-q^{s}x)~~\frac{%
[k]_{p,q}}{p^{k-m-\ell}[m]_{p,q}}\\
&=&\frac1{p^{\frac{n(n-3)}2}}\frac{[m+\ell]_{p,q}}{[m]_{p,q}}\sum\limits_{k=0}^{m+\ell-1}\left[
\begin{array}{c}
m+\ell-1 \\
k%
\end{array}%
\right] _{p,q}p^{\frac{(k+1)(k-2)}2}x^{k+1}\prod\limits_{s=0}^{m+\ell-k-2}(p^{s}-q^{s}x)\\
&=&\frac x{p^{\frac{(m+\ell-1)(m+\ell-2)}2}}\frac{[m+\ell]_{p,q}}{[m]_{p,q}}\sum\limits_{k=0}^{m+\ell-1}\left[
\begin{array}{c}
m+\ell-1 \\
k%
\end{array}%
\right] _{p,q}p^{\frac{k(k-1)}2}x^{k}\prod\limits_{s=0}^{m+\ell-k-2}(p^{s}-q^{s}x)\\
&=&\frac{[m+\ell]_{p,q}}{[m]_{p,q}}x.
\end{eqnarray*}

\item[$($iii$)$]
${B}_{m,\ell}^{p,q}(e_2;x)$
\begin{eqnarray*}
 &=&\frac1{p^{\frac{(m+\ell)(m+\ell-1)}2}}\sum\limits_{k=0}^{m+\ell}\left[
\begin{array}{c}
m+\ell\\
k%
\end{array}%
\right] _{p,q}p^{\frac{k(k-1)}2}x^{k}\prod\limits_{s=0}^{m+\ell-k-1}(p^{s}-q^{s}x)~~\frac{%
[k]_{p,q}^2}{p^{2k-2m-2\ell}[m]_{p,q}^2}\\
&=&\frac1{p^{\frac{(m+\ell)(m+\ell-5)}2}}\frac{[m+\ell]_{p,q}}{[m]_{p,q}^2}\sum\limits_{k=0}^{m+\ell-1}\left[
\begin{array}{c}
m+\ell-1 \\
k%
\end{array}%
\right] _{p,q}p^{\frac{(k+1)(k-4)}2}x^{k+1}\prod\limits_{s=0}^{m+\ell-k-2}(p^{s}-q^{s}x)~[k+1]_{p,q}\\
&=&\frac1{p^{\frac{(m+\ell)(m+\ell-5)}2}}\frac{[m+\ell]_{p,q}}{[m]_{p,q}^2}\sum\limits_{k=0}^{m+\ell-1}\left[
\begin{array}{c}
m+\ell-1 \\
k%
\end{array}%
\right] _{p,q}p^{\frac{(k+1)(k-4)}2}x^{k+1}\prod\limits_{s=0}^{m+\ell-k-2}(p^{s}-q^{s}x)~(p^k+q[k]_{p,q})\\
&=&\frac1{p^{\frac{(m+\ell)(m+\ell-5)}2}}\frac{[m+\ell]_{p,q}}{[m]_{p,q}^2}\sum\limits_{k=0}^{m+\ell-1}\left[
\begin{array}{c}
m+\ell-1 \\
k%
\end{array}%
\right] _{p,q}p^{\frac{k^2-k-4}2}x^{k+1}\prod\limits_{s=0}^{m+\ell-k-2}(p^{s}-q^{s}x)\\
&&+\frac {q[m+\ell-1]_{p,q}}{p^{\frac{(m+\ell)(m+\ell-5)}2}}\frac{[m+\ell]_{p,q}}{[m]_{p,q}^2}\sum\limits_{k=0}^{m+\ell-2}\left[
\begin{array}{c}
m+\ell-2 \\
k%
\end{array}%
\right] _{p,q}p^{\frac{(k+2)(k-3)}2}x^{k+2}\prod\limits_{s=0}^{m+\ell-k-3}(p^{s}-q^{s}x)\\
&=&\frac {p^{m+\ell-1}x}{p^{\frac{(m+\ell-1)(m+\ell-2)}2}}\frac{[m+\ell]_{p,q}}{[m]_{p,q}^2}\sum\limits_{k=0}^{m+\ell-1}\left[
\begin{array}{c}
m+\ell-1 \\
k%
\end{array}%
\right] _{p,q}p^{\frac{k(k-1)}2}x^{k}\prod\limits_{s=0}^{m+\ell-k-2}(p^{s}-q^{s}x)\\
&&+\frac {q[m+\ell-1]_{p,q}[m+\ell]_{p,q}~x^2}{[m]_{p,q}^2}\frac1{p^{\frac{(m+\ell-2)(m+\ell-3)}2}}\sum\limits_{k=0}^{m+\ell-2}\left[
\begin{array}{c}
m+\ell-2 \\
k%
\end{array}%
\right] _{p,q}p^{\frac{k(k-1)}2}x^{k}\prod\limits_{s=0}^{m+\ell-k-3}(p^{s}-q^{s}x)\\
&=&\frac{p^{m+\ell-1}[m+\ell]_{p,q}}{[m]_{p,q}^2}x+\frac{%
q[m+\ell-1]_{p,q}[m+\ell]_{p,q}}{[m]_{p,q}^2}x^{2}.
\end{eqnarray*}

\end{enumerate}
\end{proof}

\begin{lemma}\label{zx}
Let ${B}_{m,\ell}^{p,q}(.;.)$ be given by lemma \eqref{bb}, then
for any $x \in [0,1]$ and $0<q<p\leq 1$ we have the following identities

\begin{enumerate}
\item[$($i$)$] ${B}_{m,\ell}^{p,q}(e_1-1;x)=\frac{%
[m+\ell]_{p,q}}{[m]_{p,q}}x-1$

\item[$($ii$)$] ${B}_{m,\ell}^{p,q}(e_1-x;x)=\left(\frac{%
[m+\ell]_{p,q}}{[m]_{p,q}}-1\right)x $

\item[$($iii$)$] ${B}_{m,\ell}^{p,q}((e_1-x)^2;x)=\frac{p^{m+\ell-1}[m+\ell]_{p,q}}{[m]_{p,q}^2}x
+\left( 1-2\frac{[m+\ell]_{p,q}}{[m]_{p,q}}+\frac{q%
[m+\ell-1]_{p,q}[m+\ell]_{p,q}}{[m]_{p,q}^2}\right)x^2.$
\end{enumerate}
\end{lemma}

\section{\textbf{{On the convergence of $(p,q)$-Bernstein-Schurer operators }}}

Let $f \in C[0,\gamma] $, and the modulus of continuity of $f$ denoted by $%
\omega(f,\delta)$ gives the maximum oscillation of $f$ in any interval of
length not exceeding $\delta>0$ and it is given by the relation
\begin{equation*}
\omega(f,\delta)=\sup_{\mid y-x \mid \leq \delta} \mid f(y)-f(x) \mid,~~~x,y
\in [0,\gamma].
\end{equation*}
It is known that $\lim_{\delta\to 0+}\omega(f,\delta)=0$ for $f \in C[0,
\gamma]$ and for any $\delta >0$ one has
\begin{equation}\label{31}
\mid f(y)- f(x) \mid \leq \left( \frac{\mid y-x \mid}{\delta}+1\right)\omega(f,\delta).
\end{equation}

For $q \in (0,1)$ and $p \in (q ,1]$ obviously have $\lim_{m \to
\infty}[m]_{p,q}=\frac{1}{p-q}$. In order to reach to the convergence
results of the operator ${B}_{m,\ell}^{p,q}$, we take a sequence $%
q_m \in (0,1)$ and $p_m \in (q_m,1]$ such that $\lim_{m \to \infty}p_m=1$
and $\lim_{m \to \infty}q_m=1$, so we get $\lim_{m \to
\infty}[m]_{p_m,q_m}=\infty.$ 

\begin{theorem} \label{cv}
Let $p=p_m,~~q=q_m$ satisfying $0<q_m<p_m\leq 1$ such that $\lim_{m \to
\infty}p_m=1,~~\lim_{m \to \infty}q_m=1$. Then for each $f \in C[0,\ell+1]$,
\begin{equation}\label{32}
\lim_{m \to \infty}{B}_{m,\ell}^{p_m,q_m}(f;x)=f,
\end{equation}
is uniformly on $[0,1]$.
\end{theorem}

\begin{proof}
The proof is based on the well known Korovkin theorem regarding the
convergence of a sequence of linear and positive operators, so it is enough
to prove the conditions
\begin{equation*}
{B}_{m,\ell}^{p_m,q_m}((e_j;x)=x^j,~~~j=0,1,2,~~~\{\mbox{as}~
m \to \infty\}
\end{equation*}
uniformly on $[0,1]$.\newline
Clearly we have
\begin{equation*}
\lim_{m \to \infty}{B}_{m,\ell}^{p_m,q_m}(e_0;x)=1.
\end{equation*}
By taking the simple calculation we get
\begin{equation*}
\lim_{m \to \infty} \frac{[m+\ell]_{p_m,q_m}}{[m]_{p_m,q_m}}=1,~~~\mbox{as}%
~~0<q_m<p_m \leq 1.
\end{equation*}
Since $~\mbox{as}~~0<q_m<p_m \leq 1$, then we get,
\begin{equation*}
\lim_{m \to \infty} \frac{[m+\ell]_{p_m,q_m}}{[m]_{p_m,q_m}^2}=0.
\end{equation*}
Hence we have
\begin{equation*}
\lim_{m \to \infty}{B}_{m,\ell}^{p_m,q_m}(e_1;x)=x
\end{equation*}
\begin{equation*}
\lim_{m \to \infty}{B}_{m,\ell}^{p_m,q_m}(e_2;x)=x^2
\end{equation*}
\end{proof}



\begin{theorem}\label{mm}
If $f \in C[0,\ell+1]$, then
\begin{equation*}
\mid {B}_{m,\ell}^{p,q}(f;x)-f(x)\mid \leq 2{\omega}%
_f(\delta_m),
\end{equation*}
where
\begin{equation*}
\delta_m= x \bigg{|} \frac{[m+\ell]_{p,q}}{[m]_{p,q}}-1\bigg{|} + \sqrt{%
\frac{[m+\ell]_{p,q}}{[m]_{p,q}}}.\sqrt{\frac{\left(q[m+\ell-1]_{p,q}-[m+%
\ell]_{p,q}\right)x^2+ p^{(m+\ell-1)}x}{[m]_{p,q}}}.
\end{equation*}
\end{theorem}

\begin{proof}

$\mid{B}_{m,\ell}^{p,q}(f;x)-f(x)\mid$
\begin{eqnarray*}
 & \leq & \frac{1}{p^{(m+\ell)(m+\ell-1)}}
\sum_{k=0}^{m+\ell} \left[
\begin{array}{c}
m+\ell \\
k%
\end{array}%
\right] _{p,q} p^{\frac{k(k-1)}{2}}x^k
\prod_{s=0}^{m+\ell-k-1}(p^s-q^s x) \bigg{|} f \left(\frac{%
[k]_{p,q}}{p^{k-m-\ell}[m]_{p,q}}\right)-f(x)\bigg{|} \\
& \leq & \frac{1}{p^{(m+\ell)(m+\ell-1)}} \sum_{k=0}^{m+\ell} \left[
\begin{array}{c}
m+\ell \\
k%
\end{array}%
\right] _{p,q}p^{\frac{k(k-1)}{2}}x^k \prod_{s=0}^{m+\ell-k-1}(p^s-q^s x) \left(\frac{%
\bigg{|} \frac{[k]_{p,q}}{p^{k-m-\ell}[m]_{p,q}} -x\bigg{|}}{\delta}+1\right)\omega(f,\delta).
\end{eqnarray*}
By using the Cauchy inequality and lemma \eqref{bb} we have \newline
$\mid{B}_{m,\ell}^{p,q}(f;x)-f(x)\mid$
\begin{eqnarray*}
 & \leq & \left( 1+\frac{1}{\delta}\left\{\frac{1}{p^{(m+\ell)(m+\ell-1)}}\sum_{k=0}^{m+\ell} \left[
\begin{array}{c}
m+\ell \\
k%
\end{array}%
\right] _{p,q}p^{\frac{k(k-1)}{2}}x^k \left(\frac{[k]_{p,q}}{[m]_{p,q}} -x
\right)^2 \prod_{s=0}^{m+\ell-k-1}(p^s-q^s x)\right\}^{\frac{1}{2}%
}\right)\\
& \times & \left({B}_{m,\ell}^{p,q}(e_0;x)\right)^{\frac{1}{2}}\omega(f,\delta) \\
&=&  \left\{ \frac{1}{\delta}\left( {B}%
_{m,\ell}^{p,q}(e_2;x)-2x{B}_{m,\ell}^{p,q}(e_1;x) +x^2{%
B}_{m,\ell}^{p,q}(e_0;x)\right)^{\frac{1}{2}}+1 \right\}\omega(f,\delta)
\end{eqnarray*}
\begin{equation*}
=  \left\{ \frac{1}{\delta}\left( \frac{[m+\ell]_{p,q}
p^{m+\ell-1}}{[m]_{p,q}^2}x+ \left(\frac{%
[m+\ell]_{p,q}[m+\ell-1]_{p,q}}{[m]_{p,q}^2} q -2 \frac{[m+\ell]_{p,q}}{%
[m]_{p,q}}+1\right)x^2\right)^{\frac{1}{2}}+1 \right\}\omega(f,\delta)
\end{equation*}
$=  \left\{\frac{1}{\delta}\left(\left( x \left(\frac{%
[m+\ell]_{p,q}}{[m]_{p,q}}-1\right) \right)^2+ \left(\sqrt{\frac{%
[m+\ell]_{p,q}}{[m]_{p,q}}}.\sqrt{\frac{\left(q[m+\ell-1]_{p,q}-[m+%
\ell]_{p,q}\right)x^2+ p^{(m+\ell-1)}x}{[m]_{p,q}}}%
\right)^2\right)^{\frac{1}{2}}+1\right\}\omega(f,\delta)$ \newline
$\leq \left\{\frac{1}{\delta}\left( x \bigg{|} \frac{%
[m+\ell]_{p,q}}{[m]_{p,q}}-1\bigg{|} + \sqrt{\frac{[m+\ell]_{p,q}}{[m]_{p,q}}%
}.\sqrt{\frac{\left(q[m+\ell-1]_{p,q}-[m+\ell]_{p,q}\right)x^2+
p^{(m+\ell-1)}x}{[m]_{p,q}}}\right)+1\right\}\omega(f,\delta)$.%
\newline
\{by using $(a^2+b^2)^{\frac{1}{2}} \leq (\mid a \mid + \mid b \mid)$\}.%
\newline
Hence we obtain the desired result by choosing $\delta=\delta_m$.
\end{proof}

\section{ \textbf{{Direct Theorems on $(p,q)$-Bernstein-Schurer operators } }%
}

The Peetre's $K$-functional is defined by
\begin{equation*}
K_{2}(f,\delta )=\inf \left\{ \left( \parallel f-g\parallel
+\delta \parallel g^{\prime \prime }\parallel \right) :g\in \mathcal{W}%
^{2}\right\} ,
\end{equation*}%
where
\begin{equation*}
\mathcal{W}^{2}=\left\{ g\in C[0,\ell +1]:g^{\prime },g^{\prime \prime }\in
C[0,\ell +1]\right\} .
\end{equation*}%
Then there exits a positive constant $\mathcal{C}>0$ such that $K%
_{2}(f,\delta )\leq \mathcal{C}\omega _{2}(f,\delta ^{\frac{1}{2}}),~~\delta
>0$, where the second order modulus of continuity is given by
\begin{equation*}
\omega _{2}(f,\delta ^{\frac{1}{2}})=\sup_{0<h<\delta ^{\frac{1}{2}%
}}\sup_{x\in \lbrack 0,\ell +1]}\mid f(x+2h)-2f(x+h)+f(x)\mid .
\end{equation*}%

\begin{theorem}\label{cvvv}
Let $f \in C[0,\ell+1],~~g^{\prime }\in C[0,\ell+1]$ and satisfying $0<q<p
\leq 1$. Then for all $n \in \mathbb{N}$ there exits a constant $\mathcal{C}%
>0$ such that
\begin{equation*}
\biggl{|} {B}_{m,\ell}^{p,q}(f;x)-f(x)- x g^{\prime }(x)\left(
\frac{[m+\ell]_{p,q}}{[m]_{p,q}}-1\right) \biggl{|} \leq \mathcal{C}%
\omega_2 (f, \delta_m(x)),
\end{equation*}
where
\begin{equation*}
\delta_m^2(x)=\frac{[m+\ell]_{p,q}}{[m]_{p,q}^2} p^{(m+\ell-1)}x
\end{equation*}
\begin{equation*}
+ \left( \left( \frac{[m+\ell]_{p,q}}{[m]_{p,q}}-1\right)^2 +\frac{%
[m+\ell]_{p,q}}{[m]_{p,q}^2}\left( q[m+\ell-1]_{p,q}-[m+\ell]_{p,q}\right)
\right)x^2
\end{equation*}
\end{theorem}

\begin{proof}

Let $g \in \mathcal{W}^2$, then from the Taylor's expansion, we get
\begin{equation*}
g(t) = g(x)+ g^{\prime }(x)(t-x)+\int_{x}^t(t-u)g^{\prime \prime }(u)
\mathrm{d}u,~~~t \in [0,\mathcal{A}],~~\mathcal{A} >0.
\end{equation*}
Now by lemma \eqref{zx}, we have
\begin{equation*}
{B}_{m,\ell}^{p,q}(g;x) = g(x)+xg^{\prime }(x)\left( \frac{%
[m+\ell]_{p,q}}{[m]_{p,q}}-1\right)+ {B}_{m,\ell}^{p,q}\left(%
\int_{x}^t(e_1-u)g^{\prime \prime }(u) \mathrm{d}u;p,q;x\right)
\end{equation*}
\begin{eqnarray*}
\biggl{|}{B}_{m,\ell}^{p,q}(g;x)-g(x)-xg^{\prime }(x)\left(
\frac{[m+\ell]_{p,q}}{[m]_{p,q}}-1\right)\biggl{|} & \leq & {B%
}_{m,\ell}^{p,q}\left( \biggl{|} \int_{x}^t \mid (e_1-u) \mid ~ \mid g^{\prime
\prime }(u) \mid \mathrm{d}u;p,q;x \biggl{|} \right) \\
& \leq & {B}_{m,\ell}^{p,q}\left((e_1-x)^2;p,q;x \right) \parallel
g^{\prime \prime }\parallel
\end{eqnarray*}
Hence we get\newline
\newline
$\bigg{|}{B}_{m,\ell}^{p,q}(g;x)-g(x)-xg^{\prime }(x)\left(
\frac{[m+\ell]_{p,q}}{[m]_{p,q}}-1\right)\bigg{|}$\newline
$\leq \parallel g^{\prime \prime }\parallel \left((\frac{[m+\ell]_{p,q}}{%
[m]_{p,q}^2} p^{(m+\ell-1)}x +\left( \left(\frac{%
[m+\ell]_{p,q}}{[m]_{p,q}}-1 \right)^2 +\frac{[m+\ell]_{p,q}}{[m]_{p,q}^2}
\left( q[m+\ell-1]_{p,q}-[m+\ell]_{p,q}\right)\right)x^2 ~ \right)$.\newline

On the other hand we have\newline

$\biggl{|} {B}_{m,\ell}^{p,q}(f;x)-f(x)- x g^{\prime
}(x)\left( \frac{[m+\ell]_{p,q}}{[m]_{p,q}}-1\right)\biggl{|} ~\leq~ \mid
{B}_{m,\ell}^{p,q}\left((f-g);x\right)-(f-g)(x)\mid $
\begin{equation*}
+\biggl{|} {B}_{m,\ell}^{p,q}(g;x)-g(x)- x g^{\prime
}(x)\left( \frac{[m+\ell]_{p,q}}{[m]_{p,q}}-1\right) \biggl{|}.
\end{equation*}

Since we know the relation
\begin{equation*}
\mid {B}_{m,\ell}^{p,q}(f;x)\mid \leq \parallel f \parallel.
\end{equation*}
Therefore\newline
$\biggl{|} {B}_{m,\ell}^{p,q}(f;x)-f(x)- x g^{\prime
}(x)\left( \frac{[m+\ell]_{p,q}}{[m]_{p,q}}-1\right) \biggl{|} ~\leq~
\parallel f-g \parallel$\newline
$+\parallel g^{\prime \prime }\parallel \left((\frac{[m+\ell]_{p,q}}{%
[m]_{p,q}^2} p^{(m+\ell-1)}x +\left( \left(\frac{%
[m+\ell]_{p,q}}{[m]_{p,q}}-1 \right)^2 +\frac{[m+\ell]_{p,q}}{[m]_{p,q}^2}
\left( q[m+\ell-1]_{p,q}-[m+\ell]_{p,q}\right)\right)x^2 ~ \right),$\newline

Now taking the infimum on the right hand side over all $g \in \mathcal{W}%
^2$, we get
\begin{equation*}
\biggl{|} {B}_{m,\ell}^{p,q}(f;x)-f(x)- x g^{\prime }(x)\left(
\frac{[m+\ell]_{p,q}}{[m]_{p,q}}-1\right) \biggl{|} ~\leq~ \mathcal{C}
K_2 \left( f, \delta_m^2(x)\right).
\end{equation*}
In the view of the property of $K$-functional, we get
\begin{equation*}
\biggl{|} {B}_{m,\ell}^{p,q}(f;x)-f(x)- x g^{\prime }(x)\left(
\frac{[m+\ell]_{p,q}}{[m]_{p,q}}-1\right) \biggl{|} ~\leq~ \mathcal{C}
\omega_2 \left( f, \delta_m(x)\right).
\end{equation*}
This completes the proof.
\end{proof}

\begin{theorem}\label{hh}
Let $f \in C[0,\ell+1]$ be such that $f^{\prime },f^{\prime \prime }\in C[0,
\ell+1]$, and the sequence $\{p_m\}$, $\{q_m\}$ satisfying $0<q_m<p_m \leq 1$
such that $p_m \to 1,~~q_m \to 1$ and $p_m^{m}\to \alpha , ~~q_m^m \to \beta$
as $m \to \infty$, where $0\leq \alpha, \beta <1$. Then
\begin{equation*}
\lim_{m \to \infty}[m]_{p_m,q_m}\left({B}_{m,%
\ell}^{p_m,q_m}(f;x)-f(x)\right)=\frac{x(\lambda-\alpha x)}{2}f^{\prime \prime
}(x),
\end{equation*}
is uniformly on $[0,\ell+1]$, where $0< \lambda \leq 1.$
\end{theorem}

\begin{proof}
From the Taylors formula we have
\begin{equation*}
f(t)= f(x)+f^{\prime }(x)(t-x)+\frac{1}{2} f^{\prime \prime
}(x)(t-x)^2+r(t,x)(e_1-x)^2,
\end{equation*}
where $r(t,x)$ is the remainder term and $\lim_{t \to x}r(t,x)=0$,
therefore we have\newline

$[m]_{p_m,q_m}\left({B}_{m,\ell}^{p_m,q_m}(f;x)-f(x)\right)$%
\newline
$=[m]_{p_m,q_m}\left(f^{\prime }(x){B}_{m,\ell}^{p_m,q_m}%
\left((e_1-x);x\right)+\frac{f^{\prime \prime }(x)}{2}{%
B}_{m,\ell}^{p_m,q_m}\left((e_1-x)^2;x\right) +{B}%
_{m,\ell}^{p_m,q_m}(r(t,x)(t-x)^2;x)\right).$ Now by applying the
Cauchy-Schwartz inequality, we have
\begin{equation*}
{B}_{m,\ell}^{p_m,q_m}\left(r(t,x)(t-x)^2;x)\right) \leq
\sqrt{{B}_{m,\ell}^{p_m,q_m}\left(r^2(t,x);x)\right)}. \sqrt{%
{B}_{m,\ell}^{p_m,q_m}\left((t-x)^4;x)\right)}.
\end{equation*}
Since $r^2(x,x)=0$, and $r^2(t,x) \in C[0,\ell+1]$, then for from the
Theorem \ref{cv} we have
\begin{equation*}
{B}_{m,\ell}^{p_m,q_m}\left(r^2(t,x);x)\right)=r^2(x,x)=0,
\end{equation*}
which imply that
\begin{equation*}
{B}_{m,\ell}^{p_m,q_m}\left(r(t,x)(t-x)^2;x)\right)=0
\end{equation*}
\begin{equation*}
\lim_{m \to \infty}[m]_{p_m,q_m}\left({B}_{m,\ell}^{p_m,q_m}%
\left((e_1-x);x)\right)\right) =x\lim_{m \to
\infty}[m]_{p_m,q_m}\left(\frac{[m+\ell]_{p_m,q_m}}{[m]_{p_m,q_m}}-1
\right)=0
\end{equation*}
$\lim_{m \to \infty}[m]_{p_m,q_m}\left({B}%
_{m,\ell}^{p_m,q_m}\left((e_1-x)^2;x)\right)\right)$
\begin{equation*}
=x\lim_{m \to \infty}[m]_{p_m,q_m} \frac{[m+\ell]_{p_m,q_m}}{[m]_{p_m,q_m}^2}
p_m^{m+\ell-1}
\end{equation*}
\begin{equation*}
+x^2\lim_{m \to \infty}[m]_{p_m,q_m}\left( \left(\frac{[m+\ell]_{p_m,q_m}}{%
[m]_{p_m,q_m}}-1 \right)^2+\frac{[m+\ell]_{p_m,q_m}}{[m]_{p_m,q_m}^2} \left(
q_m[m+\ell-1]_{p_m,q_m}-[m+\ell]_{p_m,q_m}\right)\right)
\end{equation*}
\begin{equation*}
\lim_{m \to \infty}[m]_{p_m,q_m}\left({B}_{m,\ell}^{p_m,q_m}%
\left((e_1-x)^2;x)\right)\right)=\lambda x-\alpha x^2
=x(\lambda-\alpha x),
\end{equation*}
where $\lambda \in (0,1]$ depending on the sequence $\{p_m\}$.\newline
Hence we have
\begin{equation*}
\lim_{m \to \infty}[m]_{p_m,q_m}\left({B}_{m,%
\ell}^{p_m,q_m}(f;x)-f(x)\right)=\frac{x(\lambda-\alpha x)}{2}f^{\prime \prime
}(x).
\end{equation*}
This completes the proof.
\end{proof}

Now we give the rate of convergence of the operators ${B}%
_{m,\ell}^{p,q}(f;x) $ in terms of the elements of the usual Lipschitz class $%
Lip_{M}(\nu )$.

Let $f\in C[0,m+\ell]$, $M>0$ and $0<\nu \leq 1$. We recall that $f
$ belongs to the class $Lip_{M}(\nu )$ if the inequality
\begin{equation*}
\mid f(t)-f(x)\mid \leq M\mid t-x\mid^{\nu }~~~(t,x\in (
0,1])
\end{equation*}%
is satisfied.\newline

\begin{theorem}
Let $0<q<p\leq 1$. Then for each $f\in Lip_{M}(\nu )$ we have
\begin{equation*}
\mid{B}_{m,\ell}^{p,q}(f;x)-f(x)\mid \leq M\delta
_{m}^{\nu }(x)
\end{equation*}
where $\delta _{m}^{2}(x)=\frac{[m+\ell]_{p,q}}{[m]_{p,q}^2}
p^{(m+\ell-1)}x$
\begin{equation*}
+\left( \left(\frac{[m+\ell]_{p,q}}{[m]_{p,q}}-1 \right)^2 +\frac{%
[m+\ell]_{p,q}}{[m]_{p,q}^2} \left(
q[m+\ell-1]_{p,q}-[m+\ell]_{p,q}\right)\right)x^2.
\end{equation*}
\end{theorem}

\begin{proof}
By the monotonicity of the operators ${B}_{m,\ell}^{p,q}(f;x)$%
, we can write\newline
$\mid {B}_{m,\ell}^{p,q}(f;x)-f(x)\mid $
\begin{eqnarray*}
&\leq &{B%
}_{m,\ell}^{p,q}\left( \mid f(t)-f(x)\mid;p,q;x\right) \\
&\leq & \frac{1}{p^{(m+\ell)(m+\ell-1)}}\sum_{k=0}^{m+\ell} \left[
\begin{array}{c}
m+\ell \\
k%
\end{array}%
\right] _{p,q} p^{\frac{k(k-1)}{2}}x^k
\prod_{s=0}^{m+\ell-k-1}(p^s-q^s x) \biggl{|} f \left( \frac{%
[k]_{p,q}}{p^{k-m-\ell}[m]_{p,q}}\right)-f(x)\biggl{|} \\
&\leq & M \frac{1}{p^{(m+\ell)(m+\ell-1)}}\sum_{k=0}^{m+\ell} \left[
\begin{array}{c}
m+\ell \\
k%
\end{array}%
\right] _{p,q} p^{\frac{k(k-1)}{2}}x^k \prod_{s=0}^{m+\ell-k-1}(p^s-q^s x) \biggl{|}
\frac{[k]_{p,q}}{p^{k-m-\ell}[m]_{p,q}}-x\biggl{|}^{\nu} \\
&=& M\sum_{k=0}^{m+\ell} \left(\frac{1}{p^{(m+\ell)(m+\ell-1)}} \mathcal{P}%
_{m,\ell,k}(x) \left(\frac{[k]_{p,q}}{p^{k-m-\ell}[m]_{p,q}}-x
\right)^2\right)^{\frac{\nu}{2}} \left(\frac{1}{p^{(m+\ell)(m+\ell-1)}}\mathcal{P}_{m,\ell,k}(x)\right)^{\frac{2-\nu}{2}%
},
\end{eqnarray*}
where $\mathcal{P}_{m,\ell,k}(x)=\left[
\begin{array}{c}
m+\ell \\
k%
\end{array}%
\right] _{p,q} p^{\frac{k(k-1)}{2}}x^k
\prod_{s=0}^{m+\ell-k-1}(p^s-q^s x)$

Now applying the H\"{o}lder's inequality for the sum with $p=\frac{2}{\nu}$ and $%
q=\frac{2}{2-\nu }$ \newline
$\mid {B}_{m,\ell}^{p,q}(f;x)-f(x)\mid$
\begin{eqnarray*}
 &\leq & M
\left(\frac{1}{p^{(m+\ell)(m+\ell-1)}} \sum_{k=0}^{m+\ell} \mathcal{P}_{m,\ell,k}(x) \left(\frac{%
[k]_{p,q}}{[m]_{p,q}}-x \right)^2\right)^{\frac{\nu}{2}}
\left(\frac{1}{p^{(m+\ell)(m+\ell-1)}}\sum_{k=0}^{m+\ell} \mathcal{P}_{m,\ell,k}(x)\right)^{\frac{%
2-\nu}{2}} \\
&=&M \left( {B}_{m,\ell}^{p,q}\left((e_1-x)^2;x
\right)\right)^{\frac{\nu}{2}}
\end{eqnarray*}
Choosing $\delta :\delta _{m}(x)=\sqrt{{B}%
_{m,\ell}^{p,q}\left((e_1-x)^2;x \right)}$,\newline
we obtain
\begin{equation*}
\mid{B}_{m,\ell}^{p,q}(f;x)-f(x)\mid \leq M\delta
_{m}^{\nu }(x).
\end{equation*}
Hence, the desired result is obtained.
\end{proof}


\end{document}